\documentclass[12pt,reqno]{amsart}
\usepackage[T1]{fontenc}   %%%%%%%Please do not change
\usepackage{amssymb}
\usepackage{graphicx}
\usepackage{amsmath}
\usepackage{amsfonts,amssymb}
\usepackage{amscd}
\usepackage[active]{srcltx}
\usepackage{verbatim}

\newtheorem{theo}{Theorem}[section]
\newtheorem{thm}[theo]{Theorem}
\newtheorem{lem}[theo]{Lemma}

\newtheorem{proposition}[theo]{Proposition}

\newtheorem*{thmM}{Main Theorem}
\newtheorem*{thmS}{The Sharkovsky Theorem}

\theoremstyle{definition}
\newtheorem{dfn}[theo]{Definition}

\theoremstyle{remark}
\newtheorem{remark}[theo]{Remark}

\numberwithin{equation}{section}

\def\N{\mathbb{N}}

\newcommand{\nd}{\mathcal{ND}}
\newcommand{\nbs}{\mathcal{NBS}}
\newcommand{\n}{\mathcal{N}}
\newcommand{\orp}{\mathrm{orp}}
\newcommand{\ORP}{\mathrm{ORP}}
\newcommand{\Ovr}{\mathrm{Ovr}}

\newcommand{\Sh}{\operatorname{Sh}}
\newcommand{\Per}{\operatorname{Per}}
\newcommand{\sha}{\succ\mkern-14mu_s\;}
\renewcommand\le{\leqslant}
\renewcommand\ge{\geqslant}
\renewcommand{\rho}{\varrho}

\begin{document}

%%%%%%%%%%%%%%%%%%%%%%%%%%%%%%%%%%%%%%%%%%%%%%%%%%%%%%%%%%%%
%%%%%%%%%%%%%%%%%%%%%%%%%%%%%%%%%%%%%%%%%%%%%%%%%%%%%%%%%%%%
% This a placeholder for the TOPLOGY PROCEEDINGS logo %%%%%%
\noindent                                             %%%%%%
\begin{picture}(150,36)                               %%%%%%
\put(5,20){\tiny{Submitted to}}                       %%%%%%
\put(5,7){\textbf{Topology Proceedings}}              %%%%%%
\put(0,0){\framebox(140,34){}}                        %%%%%%
\put(2,2){\framebox(136,30){}}                        %%%%%%
\end{picture}                                        %%%%%%
%%%%%%%%%%%%%%%%%%%%%%%%%%%%%%%%%%%%%%%%%%%%%%%%%%%%%%%%%%%%
%%%%%%%%%%%%%%%%%%%%%%%%%%%%%%%%%%%%%%%%%%%%%%%%%%%%%%%%%%%%
\vspace{0.5in}

\renewcommand{\bf}{\bfseries}
\renewcommand{\sc}{\scshape}
%insert defs/styles
\vspace{0.5in}

\date{June 23, 2018; revised August 22, 2018}

\title{Forcing among patterns with no block structure}

\author[A.~Blokh]{Alexander~Blokh}

\address[Alexander~Blokh]
{Department of Mathematics\\ University of Alabama at Birmingham\\
Birmingham, AL 35294\\ USA}
\email{ablokh@math.uab.edu}

\author[M.~Misiurewicz]{Micha{\l}~Misiurewicz}

\address[Micha{\l}~Misiurewicz]
{Department of Mathematical Sciences\\ Indiana University-Pur\-due
University Indianapolis\\ 402 N. Bla\-ck\-ford Street\\
Indianapolis, IN 46202\\ USA}
\email{mmisiure@math.iupui.edu}



\dedicatory{Dedicated to the memory of our colleague and dear friend
Sergiy Kolyada}

\thanks{Research of Micha{\l} Misiurewicz was partially
supported by grant number 426602 from the Simons Foundation.}

\thanks{Research by Alexander Blokh was partially
supported by NSF grant DMS--1201450}

\thanks{Both authors would like to thank the organizers of the 52nd
  Spring Topology and Dynamical Systems Conference at Auburn
  University, during which this paper was discussed.}

\begin{abstract}
Define the following order among all natural numbers except for 2 and
1:
\[
4\gg 6\gg 3\gg \dots \gg 4n\gg 4n+2\gg 2n+1\gg 4n+4\gg\dots
\]
Let $f$ be a continuous interval map. We show that if $m\gg s$ and $f$
has a cycle with no division (no block structure) of period $m$ then
$f$ has also a cycle with no division (no block structure) of period
$s$. We describe possible sets of periods of cycles of $f$ with no
division and no block structure.
\end{abstract}

\maketitle

\section{Introduction and statement of the results}\label{s:intro}

The simplest type of limit behavior of a trajectory is periodic;
studying \emph{periodic orbits} (\emph{cycles}) is one of the central
topics in one-dimension\-al dynamics. To some extent this can be
explained by a remarkable result, the Sharkovsky Theorem, proved by
A.~N.~Sharkovsky in the 1960s (see \cite{sha64} and~\cite{shatr} for
its English translation). To state it, let us first recall the
\emph{Sharkovsky order} of the set $\N$ of positive integers:
\begin{multline*}
3\sha 5\sha 7\sha\dots\sha 2\cdot3\sha 2\cdot5\sha 2\cdot7\sha\\ \dots
\sha 2^2\cdot3\sha 2^2\cdot5\sha 2^2\cdot7\sha\dots\sha 2^2\sha 2\sha 1.
\end{multline*}
Denote by $\Sh(k)$ the set of all integers $m$ such that $k\sha m$ or
$m=k$, and by $\Sh(2^\infty)$ the set $\{1,2,4,8,\dots\}$; denote by
$\Per(f)$ the set of periods of cycles of a map $f$ (by
the \emph{period} we mean the \emph{minimal} period). Below $I$ always
denotes a closed interval.

\begin{thmS}
If $g:I\to I$ is continuous, $m\sha n$ and $m\in\Per(g)$ then
$n\in\Per(g)$ and there exists $k\in\N\cup\{2^\infty\}$ with
$\Per(g)=\Sh(k)$. Conversely, if $k\in\N\cup\{2^\infty\}$ then there
exists a continuous map $f:I\to I$ with $\Per(f) =\Sh(k)$.
\end{thmS}

The Sharkovsky Theorem is important, in particular, because it
introduces a concept of \emph{forcing relation}: it states that if
$m\sha n$ then the fact that an interval map has a cycle of period $m$
\emph{forces} the presence of a cycle of period $n$.
Thus, it shows how various ``types'' of cycles (here by
``type'' one means ``period'') force each other. Another interpretation
of the Sharkovsky Theorem is that it fully describes all possible sets
of periods of cycles of interval maps. This leads to similar problems:
(a)~how the existence of cycles of certain types forces the existence
of cycles of certain other types, and (b) what possible sets of types
of cycles an interval map may have.

For example, given a cycle $P=\{x_1<x_2<\dots <x_n\}$ of an interval
map $f$, associate with it the (cyclic) permutation $\pi$ defined by
$f(x_i)=x_{\pi(i)}$, $i=1, 2, \dots, n$. Think of $\pi$ as the type of
$P$. The family of all cycles associated to $\pi$ is called an
\emph{oriented pattern} (see~\cite{alm}). If we identify oriented
patterns obtained from each other by a flip,
we get \emph{patterns} (we denote patterns with capital letters
$A,B,\dots$). Similar to the Sharkovsky Theorem, one can ask for an
interval map $f$ (a) how cycles of certain patterns force cycles of
other patterns, and (b) what possible sets of patterns of cycles $f$
may have.

A useful way of studying patterns is by decomposing them.

\begin{dfn}[Block structure]\label{d:no-bs}
Let $\pi$ be a permutation of the set $X=\{1, \dots, n\}$. Suppose that
for some $k>1$ and $m>1$ we have $n=km$ and the permutation $\pi$ maps
sets $Y_1=\{1, \dots, m\},$ $Y_2=\{m+1, \dots, 2m\},$ $\dots$,
$Y_k=\{n-m+1, \dots, n\}$ to one another. Then sets $Y_1, \dots, Y_k$
are called \emph{blocks} and the permutation $\pi$ is said to have
\emph{block structure}; if blocks are two-point sets, $\pi$ is said to
be a \emph{doubling}. As always, similar terminology is used for
patterns and cycles. Otherwise a permutation (a pattern, a cycle) is
said to have \emph{no block structure}.
\end{dfn}

The appropriate power of the map on a block can be viewed as a kind of
renormalization of a pattern; patterns with block structure admit a
renormalization like that. Consider an important particular case.

\begin{dfn}[No division]\label{d:no-div}
Let $\pi$ be a permutation of the set $\{1, \dots, 2m\}$ such that $\pi(i)\ge
m+1$ for each $i, 1\le i\le m$ (and, hence, $\pi(i)\le m$ for each
$i\ge m+1$). Then we say that $\pi$ (and the corresponding pattern and
cycles) \emph{has division}. Otherwise $\pi$ (and the corresponding
pattern and cycles) is said to have \emph{no division}.
\end{dfn}

Observe that a pattern of period 2 has no block structure but has a
division. Therefore we will treat period 2 separately.

Consider the family  $\nbs$ of patterns with no block structure and the
family $\nd$ of all patterns with no division. A pattern \emph{with}
block structure can be studied in two steps: study the factor-pattern
obtained if each block is collapsed to a point while the order among
blocks is kept, and then study the restriction of the pattern on
blocks. On the other hand, no division patterns constitute the ``core''
in the Sharkovsky Theorem. Thus, both patterns from $\nbs$ and $\nd$ are
important. To get uniformity, we consider only patterns of periods
larger than 2; patterns of periods 1 and 2 are discussed after the Main
Theorem.

Define the following order among all natural numbers larger than 2:
\begin{equation}\label{star}
4\gg 6\gg 3\gg \dots \gg 4n\gg 4n+2\gg 2n+1\gg 4n+4\gg\dots\tag{$*$}
\end{equation}
We get it by writing  even numbers in the natural order and
inserting odd numbers $n$ after $2n$. We regard $\gg$ as a strict
ordering (that is, it is not reflexive).

Let $\n_r$ be the set of all integers $s$ with $r\gg s$ and $r$ itself.
Given an interval map $f$, let $ND(f)$ be the set of periods (larger
than 2) of all $f$-cycles with no division, and let $NBS(f)$ be the set
of periods (larger than 2) of all $f$-cycles with no block structure.

\begin{thmM}
Let $f$ be a continuous interval map. If $m\gg s$ and
$f$ has a cycle with no division (no block structure) of period $m$
then $f$ has also a cycle with no division (no block structure) of
period $s$. The following are the only possible cases, and all of them
occur.
\begin{enumerate}
\item $ND(f)=NBS(f)=\emptyset$.
\item $ND(f)=NBS(f)=\n_r$, $r\ge 3$.
\item $ND(f)=\n_{4n+2}, NBS(f)=\n_{2n+1}$, $n\ge 1$.
\end{enumerate}
\end{thmM}

Complementing this theorem, we get additional information about the
structure of cycles if $NBS(f)=\n_{2n+1}$ (see
Proposition~\ref{stefan-only} and Remark~\ref{so}).

\begin{remark}\label{onetwo}
Consider patterns of period 1 and 2. A continuous interval map always
has a fixed point, so 1 should stand at the end of the
order~\eqref{star} both for both types of patterns. The situation with
2 is more complicated. Namely, there is only one pattern of period 2,
and it has a division, but not a block structure. Thus, for no division
patterns, 2 does not occur in the order. For no block structure
patterns, 2 should stand just before 1 because, by the Sharkovsky
Theorem, if $f$ has a cycle of period larger than 1, it has also a
cycle of period 2.
\end{remark}

\begin{remark}\label{refrem}
In view of the second part of Main Theorem, the first part can be
stated in a slightly stronger fashion. Namely, let $f$ be a continuous
interval map and $m,s\ge 3$. Suppose that $f$ has a cycle of period
$m$ with \emph{no division} and either (i)~$m\gg s$, or (ii)~$m=s$ and $m$ is
not of the form $4n+2$. Then $f$ has a cycle of period $s$ with
\emph{no block structure}.
\end{remark}

\begin{remark}\label{triod}
The order~\eqref{star} is
similar to the orders present for the continuous triod map
(see~\cite{almy}) and given by $5,8,4,11,$ $14,7,\dots$ and
$7,10,5,13,16,8,\dots.$ This makes interesting connections and allows
us to look at an interval as a ``diod.''
\end{remark}

\begin{remark}\label{minor}
In~\cite{mis94} it was proved that (a) patterns from $\nbs$ of period
$2n+1$ force patterns from $\nbs$ of period $4n+4$, and (b) patterns
from $\nbs$ of period $4n$ force patterns from $\nbs$ of period $4n+2$.
However our proofs here
are much simpler (because they use the \emph{rotation theory}). The
fact that patterns with no block structure of period $4n+2$ force
patterns with no block structure of period $2n+1$ is new; the order
~\eqref{star} was mentioned in~\cite{mis94} only for unimodal maps.
Finally, in our Main Theorem we take into account not only patterns
with no block structure but also patterns with no division.
\end{remark}

\noindent \textbf{Acknowledgments.} The authors are indebted to the referee for
useful remarks.

\section{Preliminaries}\label{s:prel}

We will be using standard tools of combinatorial dynamics. The reader
that is not acquainted with them can find details for instance
in~\cite{alm}, \cite{bc} or~\cite{mini}.

In particular, we will consider forcing among patterns. It is a
partial ordering on patterns (see~\cite{bal87}). Given a
pattern $A$ we will often consider a cycle $P$ and the $P$-linear
(``connect the dots'') map $f$. Patterns forced by $A$ are then
exactly the patterns of cycles of $f$. They can be found by looking at
the Markov graph of $(f,P)$, where vertices are the $P$-basic
intervals (intervals between consecutive points of $P$) and arrows
correspond to $f$-covering (there is an arrow from $J$ to $K$ if
$K\subset f(J)$). The loops in this graph correspond to cycles of $f$
(and, hence, they determine which patterns are forced by $A$).

We will also use extensively rotation theory for interval maps. Since
it is less known, we will present its basic notions and results
(see~\cite{blo95, bm97, bm97a, bm99, bs13}). We also prove some simple
lemmas that will be necessary later.

Let $f:I\to I$ be a continuous map with a cycle $P$ of period $q>1$. Let $m$
be the number of points $x\in P$ with $f(x)-x$ and $f^2(x)-f(x)$ of
different signs. Then the pair $(m/2,q)$ is called the
\emph{over-rotation pair} of $P$ and is denoted by $\orp(P)$; the
number $m/(2q)$ is called the \emph{over-rotation number} of the cycle
$P$ and is denoted by $\rho(P)$. The set of the over-rotation pairs of
all cycles of $f$ is denoted by $\ORP(f)$. Note that the number $m$
above is even, positive, and does not exceed $q/2$. Therefore in an
over-rotation pair $(p,q)$ both $p$ and $q$ are integers and $0<p/q\le
1/2$. We call over-rotation pairs $(p, q)$
\emph{coprime} if $p$ and $q$ are coprime. Clearly, we can speak
of over-rotation pairs and over-rotation numbers of patterns and
permutations.

\begin{dfn}\label{d:orp-order}
We write $(p,q)\gtrdot (r,s)$ if $p/q<r/s$, or $p/q=r/s=m/n$ with $m$
and $n$ coprime and $p/m\sha r/m$ (clearly, $p/m,r/m\in\mathbb N$).
\end{dfn}

The next lemma relates the fact that a pattern has a block structure to
the properties of the pattern's over-rotation pair.

\begin{lem}\label{bsdiv}
If a cycle $P$ with the over-rotation pair $(k,m)$ has block structure
with $q$ points in every block, then $q$ divides both $k$ and $m$. In
particular, if $k$ and $m$ are coprime then $P$ has no block structure.
\end{lem}

\begin{proof}
Clearly, $q$ divides $m$. To see that $q$ divides $k$, observe that if
we identify each block to a point to get a cycle $Q$ of period $m/q$,
the over-rotation number of $Q$ will be the same as for $P$, i.e.,
$k/m$. If $\orp(Q)=(k',m')$, then $k'/m'=k/m$ and $m'=m/q$, so $k=k'q$.
However, $k'$ is an integer, so $q$ divides $k$.
\end{proof}

\begin{dfn}\label{d:ovr}
Let $\mathbb M$ be the set consisting of 0, all irrational numbers
between 0 and $1/2$, and all pairs $(\alpha,n)$, where $\alpha$ is a
rational number from $(0,1/2]$ and $n\in\mathbb N\cup\{2^\infty\}$.
Then for $\eta\in\mathbb M$ the set $\Ovr(\eta)$ is equal to the
following.
\begin{enumerate}
\item If $\eta$ is an irrational number or 0, then $\Ovr(\eta)$ is the
  set of all integer pairs $(p,q)$ with $\eta<p/q\le 1/2$.
\item If $\eta=(r/s,n)$ with r,s coprime, then $\Ovr(\eta)$ is the
  union of the set of all integer pairs $(p,q)$ with $r/s<p/q\le 1/2$
  and the set of all integer pairs $(mr,ms)$ with $m\in \Sh(n)$.
\end{enumerate}
\end{dfn}

In case (2) of Definition~\ref{d:ovr} if $n\ne 2^\infty$ then
$\Ovr(\eta)$ is the set of all over-rotation pairs $(p,q)$ with
$(nr,ns)\gtrdot(p,q)$, plus $(nr,ns)$ itself.

\begin{thm}[Theorem 3.1 of \cite{bm97}]\label{t:new-order}
If $f:[0,1]\to [0,1]$ is continuous, $(p,q)\gtrdot(r,s),$ and
$(p,q)\in\ORP(f),$ then $(r,s)\in\ORP(f)$. Thus, $\ORP(f)=\Ovr(\eta)$
for some $\eta\in\mathbb M$. Conversely, if $\eta\in\mathbb M$ then
there exists a continuous map $f:[0,1]\to [0,1]$ such that
$\ORP(f)=\Ovr(\eta)$.
\end{thm}

For some patterns automatically we get cycles of all periods.

\begin{dfn}[Convergent/divergent patterns]\label{d:c-d}
A pattern (cycle) of period $n$ is \emph{convergent} if for the
corresponding permutation $\pi$ there is a number $m<n$ such that
$\pi(i)>i$ for $i\le m$ and $\pi(i)<i$ for $i>m$; otherwise a pattern
(cycle) is \emph{divergent}.
\end{dfn}

Observe that $P$ is a convergent cycle of a $P$-linear map $f$ if and
only if $f$ has only one fixed point.

\begin{lem}\label{divergent}
Any divergent pattern forces a pattern with no block structure of
period $n$ for every $n>1$. Moreover, if $f$ is an interval map with a
periodic orbit of divergent pattern then $\ORP(f)=\Ovr(0)$.
\end{lem}

\begin{proof}
Let $P$ be a cycle of divergent pattern $A$. By Lemma~3.2 of~\cite{bm97},
if $f$ is an interval map with a cycle of divergent pattern
then $\ORP(f)=\Ovr(0)$, which is exactly the second claim of the lemma.
It follows that $f$ has cycles of over-rotation pair $(1, n)$ for
every $n$. By Lemma~\ref{bsdiv} these cycles have no block structure.
Considering a $P$-linear map $f$
we see that $A$ forces patterns with no block structure of any period
$n$ as desired.
\end{proof}

{}From now on we consider only convergent patterns. Then we can use an
alternative way of computing over-rotation pairs. Let $P$ be a cycle of
a convergent pattern $A$ with $\orp(P)=(p, q)$. We will always denote
by $a_P=a$ the fixed point of the $P$-linear map $f$ (we may omit the
subscript $P$ if no ambiguity is possible). Then
$(x-f(x))(f(x)-f^2(x))<0$
if and only if $x$ is mapped to the other side of $a$ under $f$. Thus,
$p$ equals the number of times when a point in $P$ is mapped from the left
of $a$ to the right of $a$ (alternatively, from the right of $a$ to the
left of $a$); $p$ can also be computed if we count the number of times
in $P$ when a points maps from one side of $a$ to the other side of
$a$, and divide this number by $2$. We can think of $p$ as a
cumulative rotation of $P$ about $a$. This interpretation
helps, in particular, in the proof of the next lemma.

\begin{lem}\label{l:rot-no-div}
If $A$ is a convergent pattern, $\rho(A)=1/2$ if and only if
$P$ has division.
\end{lem}

\begin{proof}
If $A$ has division then $\rho(A)=1/2$. Now, if $\rho(A)=1/2$, then
$\orp(A)=(n, 2n)$ for some $n$. Let $P$ be a cycle of pattern $A$, and let
$f$ be a $P$-linear map. Then $P$ has $2n$ points and all of them are
mapped from one side of $a$ to the other side. Therefore, $P$ has a
division.
\end{proof}

Another concept related to Theorem~\ref{t:new-order} is that of a
\emph{twist pattern}.

\begin{dfn}[Twist patterns]\label{d:twist}
A pattern of over-rotation number $\rho$ is \emph{twist} if
it does not force any other pattern of over-rotation number $\rho$; we
use the same terminology for cycles and permutations.
\end{dfn}

By Lemma~\ref{divergent} a twist cycle must be convergent. In
particular, if $P$ is a twist cycle then the $P$-linear map has a
unique fixed point.

\begin{lem}[\cite{bm97a, bm99}]\label{l:twist}
Let $P$ be a twist cycle $P$ of the $P$-linear map $f$. Then, if
points $u, v\in P$ lie on the same side of $a$, map to the same side
of $a$, and $u$ is farther away from $a$ than $v$, then $f(u)$ is
farther away from $a$ than $f(v)$.
\end{lem}

\section{Proof of Main Theorem}\label{s:main}

We start by recalling the definition of a well known \v Stefan pattern.

\begin{dfn}[\v Stefan pattern]\label{d:stefan}
Consider the cyclic permutation $\sigma:\{1, 2, \dots, 2n+1\}\to \{1,
2, \dots, 2n+1\}$ ($n\ge 1$), defined as follows:
\begin{itemize}
\item $\sigma(1)=n+1$;
\item $\sigma(i)=2n+3-i$,\ \ if\ \ $2\le i\le n+1$;
\item $\sigma(i)=2n+2-i$,\ \ if\ \ $n+2\le i\le 2n+2$.
\end{itemize}
Then the pattern of this cyclic permutation is called the \emph{\v
  Stefan pattern}, and any cycle of this pattern is said to be a
\emph{\v Stefan cycle}.
\end{dfn}

The importance of those patterns is due to the following fact.

\begin{thm}[\cite{ste77}]\label{t:stefan}
Any pattern of period $2n+1$ forces the \v Stefan pattern of period
$2n+1$. Moreover, if a continuous interval map $f$ has a cycle of
period $2n+1$ and no cycles of period $2k+1$ with $1\le k<n$, then
every cycle of $f$ of period $2n+1$ is \v Stefan.
\end{thm}

Now we prove some preliminary results. If $P$ is a cycle of
period $n>1$ then for each point $x\in P$ we consider
\emph{germs} at $x$, i.e., small intervals with $x$ as one of the
endpoints. Each point of $P$ has two germs, except the leftmost and
rightmost points, which have one germ each. There is a natural map
induced on the set of germs by the $P$-linear map $f$, and if we start
at the germ of the leftmost point (or the rightmost point), we get
back exactly after $n$ applications of this map. Each germ is
contained in a $P$-basic interval, so this loop of germs gives us a
loop of $P$-basic intervals. These loops are called the
\emph{fundamental loop of germs} and the \emph{fundamental loop of
intervals}. Both loops correspond to the original periodic orbit $P$
Thus, the fundamental loop of intervals contains both
the leftmost and the rightmost $P$-basic intervals. Observe that, by
Lemma~\ref{l:twist}, if $P$ is a twist cycle then any germ at $x\in P$
that points toward $a$ maps to the germ at $f(x)\in P$ that points
toward $a$ too. Thus, if $P$ is a twist cycle, then the vertices of the
fundamental loop of germs form the set of germs pointing toward $a$.

In what follows we use the following notation. Denote by $I=[b_l, b_r]$
the $P$-basic interval containing the point $a$. Observe that the arrow
$I\to I$ is a part of the Markov graph $G$ of $P$. It follows that $I$
is repeated in the fundamental loop of intervals of $P$ twice while all
other $P$-basic intervals are repeated there once. Consider the set
$P'=P\cup\{a\}$. Though the germs at points of $P$ stay the same
whether we consider $P$ or $P'$, there is a change concerning
$P'$-basic intervals versus $P$-basic intervals: $I$ is now replaced by
two $P'$-basic intervals, $I_l=[b_l, a]$ and $I_r=[a, b_r]$. Notice
that the arrows $I_l\to I_r$ and $I_r\to I_l$ are in the Markov graph
of $P'$. Clearly, a germ at a point of $P'$ is contained in a
well-defined $P'$-basic interval. Hence the fundamental loop of germs
of $P$ gives rise to the \emph{fundamental loop of $P'$-basic
  intervals} and the \emph{fundamental loop of germs of $P'$}. Thus,
we get the following lemma.

\begin{lem}\label{fl}
If $P$ is a twist cycle of period larger than $1$ then the fundamental
loop of intervals of $P'$ passes exactly once through every $P'$-basic
interval.
\end{lem}

Now we investigate twist cycles close to the fixed point.

\begin{lem}\label{closetoa}
If $P$ is a twist cycle of period $n>2$ of a $P$-linear map $f$ then
at least one of the points $b_l,b_r$ is the image of a point of
$P$ that lies on the same side of $a$.
\end{lem}

\begin{proof}
Suppose that $b_l=f_P(c_l)$, $b_r=f(c_r)$, where
$c_l,c_r\in P$, and $c_r\le b_l<b_r\le c_l$. Since $n>2$,
either $c_r<b_l$ or $b_r<c_l$. We may assume that
$c_r<b_l$. However, then $c_r<b_l<a<b_r=f(c_r)<f(b_l),$
which contradicts Lemma~\ref{l:twist}.
\end{proof}

Twist patterns force other patterns with specific properties.

\begin{proposition}\label{mplus2}
If $P$ is a twist cycle of the $P$-linear map $f$ and $P$ has
over-rotation pair $(k,m)$ and over-rotation number $\rho(P)<\frac12$,
then $f$ has a cycle of over-rotation pair $(k+1,m+2)$, which
is not a doubling.
\end{proposition}

\begin{proof}
By Lemma~\ref{closetoa}, we may assume that $b_l=f(c_l)$ for some
$c_l\in P$ with $c_l<b_l$. Let $L$ be the fundamental loop of
intervals of $P'$. By Lemma~\ref{fl} it passes through $I_l$ exactly
once, so we can insert into $L$ the two arrows, $I_l\to I_r\to
I_l$, at that place. Denote by $M$ the loop of length $m+2$ obtained
in this way, and by $Q$ a corresponding periodic orbit of $f$. By the
construction, each $P'$-basic interval contains one element of $Q$,
except $I_l$ and $I_r$, which contain two elements each. This, in
particular, shows that the period of $Q$ is $m+2$.

Let $x\in Q$ be the point that belongs to the $P'$-basic interval
whose left endpoint is $c_l$. Then $f(x)\in I_l$, $f^2(x)\in
I_r$, and $f^3(x)\in I_l$. Since the fixed point $a$ is
repelling (because the interval $[b_l, b_r]$ is mapped linearly
onto a larger interval), we get $x<f^3(x)<f(x)<a<f^2(x)$. The
other point of $Q$ which is in $I_r$, is to the right of $f^2(x)$,
because otherwise its image would be the third point of $Q$ in $I_l$
(and by the construction there are two points of $Q$ in $I_l$ and
two points of $Q$ in $I_r$).

If $Q$ is a doubling, then $f(x)$ is paired with $f^3(x)$ or
$f^2(x)$. The first option is impossible because if it holds then
the pair of points mapped to the pair $\{f^3(x), f(x)\}$ must be
the pair $\{x, f^2(x)\}$ and the points $x$ and $f^2(x)$ are not
consecutive in space. The second option is impossible because then the
image pair $\{f^2(x), f^3(x)\}$ consists of points that are not
consecutive in space. Thus, $Q$ is not a doubling.

Finally, since we added two points that are mapped onto the opposite
side of $a$, and the rest of the points of $Q$ are mapped like the
analogous points of $P$, the over-rotation pair of $Q$ is $(k+1,m+2)$.
\end{proof}

{}From Lemma~\ref{bsdiv} and Proposition~\ref{mplus2} we get the
following lemma.

\begin{lem}\label{mp2}
Any pattern $A$ with $\orp(A)=(ks,ms)$ where $k$ and $m$ are coprime, and
$\rho(A)=k/m<1/2$ forces a pattern of over-rotation
pair $(k+1,m+2)$, which is not a doubling. In particular, a pattern
$A$ of over-rotation pair $(2n-1, 4n)$ forces a pattern of
over-rotation pair $(2n,4n+2)$ which has no block structure.
\end{lem}

\begin{proof}
Let $A$ be a pattern with $\orp(A)=(ks,ms)$ where $k$ and $m$ are coprime, and
$\rho(A)=k/m<1/2$. By Theorem~\ref{t:new-order} $A$
forces a twist pattern $A'$ of over-rotation pair $(k, m)$.
By Lemma~\ref{mplus2}, $A'$ forces a pattern of over-rotation pair
$(k+1, m+2)$ as desired. Now, let the over-rotation pair of $A$ be
$(2n-1, 4n)$. %Since $2n-1$ and $4n$ are coprime,
By the above $A$ forces a pattern $B$ of over-rotation pair $(2n,
4n+2)$ that is not a doubling. Let us show that $B$ has no block
structure. Indeed, the only common divisor of $2n$ and $4n+2$ is $2$.
Hence by Lemma~\ref{bsdiv} the only way $B$ can have block structure is
when $B$ is a doubling, a contradiction.
\end{proof}

In what follows we will use the notation below: for every $m>2$ set
\[
\eta(m)=\begin{cases}
(s-1,2s)&\mbox{if\ \ } m=2s,\\
(n,2n+1)&\mbox{if\ \ } m=2n+1,
\end{cases}
\]
In particular, $\eta(4n)=(2n-1, 4n)$ and $\eta(4n+2)=(2n, 4n+2)$.

\begin{proposition}\label{stefan-only}
Let $n\ge 1$; then the following claims hold.

\begin{enumerate}

\item Let $g$ be a continuous interval map. Assume that $g$ has a
    cycle of period $2n+1$ with no block structure, and all cycles
    of $g$ of periods $2k+1$ with $1\le k<n$ have block structure.
    Then all cycles of $g$ of period $2n+1$ are \v Stefan and $g$ has no cycles of periods $2k+1$ with $1\le k<n$.

\item Let $f$ be a continuous interval map. Assume that $f$ has a
cycle of period $4n+2$ and no division, but all cycles of $f$ of period
$4n+2$ have block structure. Then all cycles of $f$ of period
$4n+2$ and no division are doublings of a \v Stefan cycle.

\end{enumerate}

\end{proposition}

\begin{proof}
(1) Suppose that $g$ has a cycle of period $2n+1$ which is not \v Stefan.
Then, by Theorem~\ref{t:stefan}, $g$ has a \v Stefan cycle of period
$2k+1$ for some $k$ with $1\le k<n$. By inspection, \v Stefan patterns
have no block structure, so we get a contradiction. Moreover, the same argument
shows that $g$ does not have cycles of periods $2k+1$ with $1\le k<n$ at all.

(2) Let $P$ be a cycle of $f$ of period $4n+2$ and no
division. Then, by Lemma~\ref{l:rot-no-div}, the over-rotation number of $P$ is
less than $\frac12$. It follows that the over-rotation pair of $P$ must be
$(2n, 4n+2)$, as otherwise
it is at most $\frac{2n-1}{4n+2}<\frac{2n-1}{4n}$, so, by
Theorem~\ref{t:new-order} and Lemma~\ref{mp2}, $f$ would have a cycle
of period $4n+2$ with no block structure. Thus, since $P$ has block
structure and the greatest common divisor of $2n$ and $4n+2$ is
$2$, the cycle $P$ is a doubling over a cycle, say, $Q$ of period $2n+1$.
If $Q$ is not \v Stefan, then by~(1) there must exist a cycle of $f$ of period $2n-1$.
Since $\frac{n-1}{2n-1}<\frac{2n-1}{4n}$, it again follows from
Theorem~\ref{t:new-order} and Lemma~\ref{mp2} that $f$ has a cycle
of period $4n+2$ with no block structure, a contradiction. Hence $Q$ is a \v Stefan cycle.
\end{proof}

We are ready to prove our Main Theorem. By Lemma~\ref{l:rot-no-div},
in the proof we can consider only convergent patterns.

\begin{proof}[Proof of Main Theorem]
Recall that because we are excluding the pattern of period 2, each
pattern with no block structure has no division. By
Lemma~\ref{l:rot-no-div}, patterns with no division have over-rotation
numbers less than $1/2$. Each integer larger than 2 is of one of the
three forms: $2n+1$, $4n$, $4n+2$, with $n\ge 1$. The largest possible
over-rotation numbers smaller than $1/2$ for patterns of those periods
are, respectively, $\frac{n}{2n+1}$, $\frac{2n-1}{4n}$,
$\frac{2n}{4n+2}$. Those numbers are ordered as follows:
\[
\dots<\frac{2n-1}{4n}<\frac{2n}{4n+2}=\frac{n}{2n+1}<\frac{2n+1}{4n+4}
<\dots.
\]
Thus, by the definition of the order $\gtrdot$, we get the following
order among over-rotation pairs associated with these over-rotation
numbers:
\[
\dots\gtrdot\eta(4n)\gtrdot\eta(4n+2)\gtrdot\eta(2n+1)\gtrdot
\eta(4n+4)\gtrdot\dots\tag{$**$}
\]

Observe that the over-rotation pairs $(2n-1, 4n)$ and $(n, 2n+1)$ are
coprime; on the other hand, the over-rotation pair $(2n, \break 4n+2)$ is not
coprime as the greatest common divisor of $2n$ and $4n+2$ is $2$.

Let $f$ be a continuous interval map. If all cycles of $f$ have
division then all cycles of $f$ have block structure as division is a
particular case of block structure. This means that case~(1) of Main
Theorem takes place. To proceed with less trivial cases of Main
Theorem, fix two integers, $m>2$ and $s$ such that $m\gg s$.

Consider first the case of cycles with no division. Assume that $f$
has a cycle $P$ of period $m>2$ with no division. This cycle has
over-rotation number less than $1/2$, so by Theorem~\ref{t:new-order}
the map $f$ has a cycle of over-rotation pair $\eta(m)$. If $m\gg s$
then $\eta(m)\gtrdot \eta(s)$, so again by Theorem~\ref{t:new-order},
$f$ has a cycle $Q$ of over-rotation pair $\eta(s)$. Since its
over-rotation number is smaller than $1/2$, $Q$ has no division. In
other words, if $f$ has a cycle of period $m$ with no division and
$m\gg s$ then $f$ must have a cycle of period $s$ with no division.
This proves, for cycles with no division, the first statement of Main
Theorem.

Now, assume that $f$ has a cycle $P$ of period $m>2$ with no block
structure. Then, in particular, $P$ has no division. As before,
this implies that $f$ has a cycle of over-rotation pair $\eta(m)$
and, again, $f$ has some cycles of over-rotation pair $\eta(s)$.
To prove the first statement of Main Theorem for cycles with no block
structure we need to show that a cycle of over-rotation pair
$\eta(s)$, forced by $P$, can be chosen to be with no block structure. By
Lemma~\ref{bsdiv} and by the analysis of over-rotation pairs
$\eta(4n),$ $\eta(4n+2),$ and $\eta(2n+1),$ any cycle $Q$ of
over-rotation pair $\eta(s)$ with $s=4n$ or $s=2n+1$
automatically has no block structure.
If $s=4n+2$, then either $m=4n$ or $m\gg 4n$, so
$f$ must have a cycle of over-rotation pair $\eta(4n)=(2n-1,4n)$. Then by
Lemma~\ref{mp2} the map $f$ must have a cycle $Q$ of
over-rotation pair $(2n, 4n+2)$ and no block structure. This
completes the proof of the first statement of Main Theorem for cycles
with no block structure.

This also proves that $ND(f)$ and $NBS(f)$ are either empty or of the
form $\n_r$, i.e., $ND(f)=\n_{r_{nd}}$ and $NBS(f)=\n_{r_{nbs}}$ for
some numbers $r_{nd}$ and $r_{nbs}$. Consider all the
cases in more detail. Any cycle with no block structure has no
division. Hence in general $ND(f)\supset NBS(f)$, so
$r_{nd}\gg r_{nbs}$ or $r_{nd}=r_{nbs}$. If $r_{nd}=2n+1$ then, as before, $f$ must have a
cycle of over-rotation pair $\eta(2n+1)$ which is coprime. It follows
that this cycle has no block structure and, hence, in this case
$ND(f)=NBS(f)=\n_{2n+1}$. This covers case~(2) of Main Theorem for $r=2n+1$. If
$r_{nd}=4n$, then, again, $f$ must have a cycle of over-rotation pair
$\eta(4n)$ which is coprime, this cycle has no block structure, and
$ND(f)=NBS(f)=\n_{4n}$. This covers case~(2) of Main Theorem for $r=4n$.

Suppose now that $r_{nd}=4n+2$. Then $f$ must have a cycle of
over-rotation pair $\eta(4n+2)$. If $f$ has a cycle of over-rotation
pair $\eta(4n+2)$ with no block structure, then $ND(f)=NBS(f)=
\n_{4n+2}$, which corresponds to case~(2) of Main Theorem for $r=4n+2$. Suppose now
that all cycles of over-rotation pair $\eta(4n+2)$ have block
structure. Then, while $ND(f)=\n_{4n+2}$, the set $NBS(f)$ is strictly
smaller than $ND(f)$. The first statement of Main Theorem implies that
$f$ has a point of period $2n+1$; we may assume that its over-rotation
pair is $\eta(2n+1)$, which is coprime, so the corresponding periodic
orbit has no block structure. We conclude that in this case
$ND(f)=\n_{4n+2}$ while $NBS(f)=\n_{2n+1}$. This covers case~(3) of
Main Theorem.

To prove that all cases (1)-(3) can occur, we first note that a
constant map is an example for case~(1).

To give an example of case~(2) for a given $r\ge 3$, we observe that
there exists a pattern of period $r$ with no block structure, and
there are only finitely many such patterns. Since the forcing relation
is a partial order, there is a pattern $A$ of period $r$ with no block
structure which is minimal, in the sense that it forces no other such
pattern. Let $P$ be a cycle of pattern $A$ of the $P$ linear map $f$.
By the part of the theorem already proven, $\n_r\subset NBS(f)$.

If $f$ has a cycle $Q$ of period $m\ge 3$, $m\notin\n_r$, with no
block structure, then $m\gg r$, so the pattern $B$ of $Q$ forces some
pattern $A'$ of period $r$ with no block structure, and thus, $f$ has
a cycle of pattern $A'$. Since $f$ is $P$-linear, it follows that $A$
forces $A'$, and by minimality of $A$ we get $A'=A$. Hence $A$ forces
$B$ and $B$ forces $A$, a contradiction. This proves that
$NBS(f)=\n_r$.

By the part of the theorem already proven, either~(2) holds or
$r=2n+1$ and $ND(f)=\n_{2r}$. In the latter case, by
Proposition~\ref{stefan-only}, $A$ is a \v Stefan pattern that forces
its own doubling which is impossible. This proves that~(2) holds.

Finally, to give an example of case~(3) for a given $n\ge 1$, we take
a $P$-linear map $f$, where the pattern $A$ of $P$ is a doubling of
the \v Stefan pattern of period $2n+1$. By the theorems on forcing
extensions of patterns (see~\cite{alm}), if $B$ is a pattern and $A$
is a doubling of $B$, then $A$ forces $B$ and the only pattern forced
by $A$ but not by $B$ is $A$ itself. Since the \v Stefan pattern does
not force any other pattern of the same period,~(3) holds.
\end{proof}

\begin{remark}\label{so}
It follows from Main Theorem and Proposition~\ref{stefan-only} that
the case~(3) of Main Theorem occurs if and only if a cycle of period
$4n+2$ with no division exists and all such cycles are doublings of \v
Stefan cycles.
\end{remark}

\end{document}